\documentclass{amsart}
\usepackage{amssymb}
\usepackage{latexsym}
\usepackage{hyperref}

\newcommand{\R}{{\mathbb R}}
\newcommand{\C}{{\mathbb C}}

\newcommand{\T}{{\mathbb T}}
\newcommand{\Z}{{\mathbb Z}}
\newcommand{\D}{{\mathbb D}}
\newcommand{\U}{{\mathbb U}}

\newtheorem{theorem}{Theorem}

\begin{document}

\title[Uniformly Positive Lyapunov Exponents]{Almost Periodic Szeg\H{o} Cocycles
with Uniformly Positive Lyapunov Exponents}

\author[D.\ Damanik]{David Damanik}

\address{Department of Mathematics, Rice University, Houston, TX~77005, USA}

\email{damanik@rice.edu}

\urladdr{\href{http://www.ruf.rice.edu/~dtd3/}{http://www.ruf.rice.edu/~dtd3/}}

\author[H.\ Kr\"uger]{Helge Kr\"uger}

\address{Department of Mathematics, Rice University, Houston, TX~77005, USA}

\email{helge.krueger@rice.edu}

\urladdr{\href{http://math.rice.edu/~hk7/}{http://math.rice.edu/~hk7/}}

\thanks{D.\ D.\ was supported in part by NSF grant
DMS--0653720.}

\date{\today}

\begin{abstract}
We exhibit examples of almost periodic Verblunsky coefficients for
which Herman's subharmonicity argument applies and yields that the
associated Lyapunov exponents are uniformly bounded away from
zero.
\end{abstract}

\maketitle

\section{Introduction}

Suppose that $(\Omega,\mu)$ is a probability measure space and $T
: \Omega \to \Omega$ is ergodic with respect to $\mu$. A
measurable map $A : \Omega \to \mathrm{GL} (2,\C)$ gives rise to a
so-called cocycle, which is a map from $\Omega \times \C^2$ to
itself given by $(\omega,v) \mapsto (T \omega,A(\omega)v)$. This
map is usually denoted by the same symbol. When studying the
iterates of the cocycle, the following matrices describe the
dynamics of the second component:
$$
A_n(\omega) = A(T^{n-1} \omega) \cdots A(\omega ).
$$

Assuming $\log \| A \| \in L^1(\mu)$ and
$$
\inf_{n \ge 1} \frac{1}{n} \int_\Omega \log \| A_n(\omega)\| \,
d\mu(\omega) > - \infty,
$$
then, by Kingman's subadditive ergodic theorem, the following
limit exists,
\begin{equation}\label{MultET}
\gamma = \lim_{n \to \infty} \frac{1}{n} \int_\Omega \log \|
A_n(\omega)\| \, d\mu(\omega),
\end{equation}
and we have
$$
\gamma = \lim_{n \to \infty} \frac{1}{n}  \log \| A_n(\omega)\|
$$
for $\mu$-almost every $\omega \in \Omega$. The number $\gamma$ is
called the Lyapunov exponent of $A$.

We will be interested in the particular case of Szeg\H{o}
cocycles, which arise as follows. Denote the open unit disk in
$\C$ by $\D$. For a measurable function $f : \Omega \to \D$ with
$$
\int_\Omega \log (1 - |f(\omega)|) \, d\mu(\omega) > - \infty
$$
and $z \in \partial \D$, the cocycle $A^z : \Omega \to \U (1,1)$
is given by
\begin{equation}\label{sc}
A^z(\omega) = (1 - |f(\omega)|^2)^{-1/2} \left( \begin{array}{cc}
z & - \overline{f(\omega)} \\ - f(\omega) z & 1 \end{array}
\right).
\end{equation}
The Lyapunov exponent of $A^z$ will be denoted by $\gamma(z)$. The
complex numbers $\alpha_n(\omega) = f(T^n \omega)$, $n \ge 0$,
appearing in $A^z(T^n \omega)$ are called Verblunsky coefficients.

Szeg\H{o} cocycles play a central role in the analysis of
orthogonal polynomials on the unit circle with ergodic Verblunsky
coefficients; compare Simon's recent monograph \cite{S,S2} (see in
particular \cite[Section~10.5]{S2} for more information on
Lyapunov exponents of Szeg\H{o} cocycles). One of the major themes
of \cite{S,S2} is to work out in detail the close analogy between
the spectral analysis of Jacobi matrices, or more specifically
discrete one-dimensional Schr\"odinger operators, and that of CMV
matrices. Indeed, a large portion of the second part, \cite{S2},
is devoted to carrying over results and methods from the
Schr\"odinger and Jacobi setting to the OPUC setting.

Sometimes the transition is straightforward and sometimes it is
not. As discussed in the remarks and historical notes at the end
of \cite[Section~10.16]{S2}, one of the results that Simon did not
manage to carry over is Herman's result on uniformly positive
Lyapunov exponents for a certain class of almost periodic
Schr\"odinger cocycles \cite{H} (see also
\cite[Section~10.2]{CFKS}), which is proved by a beautiful
subharmonicity argument and which is at the heart of many of the
recent, far more technical, advances in the area of Schr\"odinger
operators with almost periodic potentials; see Bourgain's book
\cite{B} and references therein.

Our goal in this note is to present one-parameter families of
almost periodic Szeg\H{o} cocycles for which we prove uniformly
positive Lyapunov exponents using Herman's argument for an
explicit region of parameter values.

\section{Examples with Uniformly Positive Lyapunov Exponents}

Consider the $1$-torus $\T = \R / \Z$ equipped with Lebesgue measure
and $\Z_2$ equipped with the probability measure that assigns the
weight $\frac12$ to each of $0$ and $1$. Let $\Omega = \T \times
\Z_2$ be the product space and $\mu$ the product measure. Fix some
irrational $\alpha \in \T$. The transformation $T : \Omega \to
\Omega$ is given by $T(\theta,j) = (\theta + \alpha , j + 1)$. It is
readily verified that $T$ is ergodic with respect to $\mu$.

For $\varepsilon \in (0,1)$ and $k \in \Z \setminus \{ 0 \}$, we
define $f : \Omega \to \D$ by
\begin{equation}\label{fdef}
f(\theta,j) = \begin{cases} (1 - \varepsilon^2)^{1/2} e^{2 \pi i k
\theta} & j = 0, \\ (1 - \varepsilon^2)^{1/2} e^{- 2 \pi i k \theta}
& j = 1. \end{cases}
\end{equation}
Clearly, $f$ is measurable function from $\Omega$ to $\D$ and
satisfies $\log [(1 - |f|^2)^{-1/2}] \in L^1(\mu)$. Thus, the
Lyapunov exponents $\gamma(z)$, $z \in \partial \D$ exist and we
wish to bound them from below.

\begin{theorem}\label{t.1}
For $(\Omega,\mu,T)$ as above and $f$ given by \eqref{fdef}, we
have the estimate
$$
\inf_{z \in \partial \D} \gamma(z) \ge \log \frac{(1 -
\varepsilon^2)^{\frac{1}{2}}}{\varepsilon}.
$$
In particular, if $\varepsilon \in (0,\frac{1}{\sqrt{2}})$, the
Lyapunov exponent $\gamma (\cdot)$ is uniformly positive on
$\partial \D$.
\end{theorem}

\begin{proof}
We consider the case $k > 0$; the case $k < 0$ is completely
analogous. Fix any $z \in \partial \D$. By the definition \eqref{sc}
of $A^z$ and the definition \eqref{fdef} of $f$, we have
$$
A^z(\theta,j) = \begin{cases} \varepsilon^{-1} \begin{pmatrix} z &
- (1 - \varepsilon^2)^{1/2} e^{- 2 \pi i k \theta} \\ - (1 -
\varepsilon^2)^{1/2} e^{2 \pi i k \theta} z & 1
\end{pmatrix} & j = 0, \\ \varepsilon^{-1} \begin{pmatrix} z & - (1 - \varepsilon^2)^{1/2}
e^{2 \pi i k \theta} \\ - (1 - \varepsilon^2)^{1/2} e^{- 2 \pi i k
\theta} z & 1 \end{pmatrix} & j = 1.
\end{cases}
$$
Let us conjugate these matrices as follows
(cf.~\cite[Equation~(4.10)]{GHMT}). Define
$$
C^z(\theta,j) = \begin{cases} \begin{pmatrix} 0 & 1 \\ 1 & 0 \end{pmatrix} & j = 0, \\
\begin{pmatrix} z^{1/2} & 0 \\ 0 & z^{-1/2} \end{pmatrix} & j = 1.
\end{cases}
$$
For $j = 0$, we have
\begin{align*}
\varepsilon C^z(\theta,j) & A^z(\theta,j) C^z(\theta,j - 1)^{-1} = \\
& = \begin{pmatrix} 0 & 1 \\ 1 & 0 \end{pmatrix} \begin{pmatrix} z
& - (1 - \varepsilon^2)^{1/2} e^{- 2 \pi i k \theta} \\ - (1 -
\varepsilon^2)^{1/2} e^{2 \pi i k \theta} z & 1
\end{pmatrix} \begin{pmatrix} z^{- 1/2} & 0 \\ 0 & z^{1/2}
\end{pmatrix}\\
& = \begin{pmatrix} - (1 - \varepsilon^2)^{1/2} e^{2 \pi i k
\theta} z & 1 \\ z & - (1 - \varepsilon^2)^{1/2} e^{- 2 \pi i k
\theta} \end{pmatrix} \begin{pmatrix} z^{- 1/2} & 0 \\ 0 & z^{1/2}
\end{pmatrix} \\
& = \begin{pmatrix} - (1 - \varepsilon^2)^{1/2} e^{2 \pi i k
\theta} z^{1/2} & z^{1/2} \\ z^{1/2} & - (1 - \varepsilon^2)^{1/2}
e^{- 2 \pi i k \theta} z^{1/2} \end{pmatrix},
\end{align*}
while for $j = 1$, we have
\begin{align*}
\varepsilon C^z(\theta,j) & A^z(\theta,j) C^z(\theta,j - 1)^{-1} = \\
& = \begin{pmatrix} z^{1/2} & 0 \\ 0 & z^{- 1/2}
\end{pmatrix} \begin{pmatrix} z & - (1 - \varepsilon^2)^{1/2} e^{2 \pi i
k \theta} \\ - (1 - \varepsilon^2)^{1/2} e^{- 2 \pi i k \theta} z
& 1
\end{pmatrix} \begin{pmatrix} 0 & 1 \\ 1 & 0 \end{pmatrix}\\
& = \begin{pmatrix} z^{1/2} & 0 \\ 0 & z^{- 1/2}
\end{pmatrix} \begin{pmatrix} - (1 - \varepsilon^2)^{1/2} e^{2 \pi i k
\theta} & z \\ 1 & - (1 - \varepsilon^2)^{1/2} e^{- 2 \pi i k
\theta} z \end{pmatrix}  \\
& = \begin{pmatrix} - (1 - \varepsilon^2)^{1/2} e^{2 \pi i k
\theta} z^{1/2} & z^{3/2} \\ z^{- 1/2} & - (1 -
\varepsilon^2)^{1/2} e^{- 2 \pi i k \theta} z^{1/2} \end{pmatrix},
\end{align*}
Thus,
\begin{equation}\label{conj}
C^z(\theta,j) A^z(\theta,j) C^z(\theta,j - 1)^{-1} =
\frac{z^{1/2}}{\varepsilon}
\begin{pmatrix} - (1 - \varepsilon^2)^{1/2} e^{2 \pi i
k \theta} & z^{j} \\ z^{-j} & - (1 - \varepsilon^2)^{1/2} e^{- 2
\pi i k \theta} \end{pmatrix}.
\end{equation}
We have $A^z_n(\theta,j) = A^z( \theta + (n-1) \alpha , j + n - 1)
\cdots A^z(\theta,j)$, which, by \eqref{conj}, is equal to
$$
C^z(\theta,j + n - 1)^{-1} \prod_{m = n-1}^0 \left(
\frac{z^{1/2}}{\varepsilon}
\begin{pmatrix} - (1 - \varepsilon^2)^{1/2} e^{2 \pi i
k (\theta + m \alpha)} & z^{(j+m \!\!\!\! \mod 2)} \\ z^{-(j + m
 \!\!\!\! \mod 2)} & - (1 - \varepsilon^2)^{1/2} e^{- 2 \pi i k (\theta + m
\alpha)}
\end{pmatrix} \right) C^z(\theta,j - 1).
$$
Since $C^z(\theta,j)$ is always unitary and $w = e^{2 \pi i
\theta}$ and $z^{1/2}$ both have modulus one, we find that
\begin{align*}
\| A^z_n(\theta,j) \| & = \varepsilon^{-n}\left\| \prod_{m = n-1}^0
\begin{pmatrix} -(1 - \varepsilon^2)^{1/2} e^{2 \pi i k (\theta + m
\alpha)} & z^{(j+m  \!\!\!\! \mod 2)} \\ z^{-(j+m \!\!\!\! \mod
2)} & -(1 - \varepsilon^2)^{1/2}
e^{- 2 \pi i k (\theta + m \alpha)} \end{pmatrix} \right\| \\
& = \varepsilon^{-n}\left\| \prod_{m = n-1}^0 \begin{pmatrix} -(1
- \varepsilon^2)^{1/2} e^{2 \pi i k (2\theta + m \alpha)} &
z^{(j+m \!\!\!\! \mod 2)} e^{2 \pi i k \theta} \\ z^{-(j+m
\!\!\!\! \mod 2)} e^{2 \pi i k \theta} & -(1 -
\varepsilon^2)^{1/2} e^{- 2 \pi
i k m \alpha} \end{pmatrix} \right\| \\
& = \varepsilon^{-n} \left\| \prod_{m = n-1}^0 \begin{pmatrix} -(1
- \varepsilon^2)^{1/2} e^{2 \pi i k m \alpha} w^{2k} & z^{(j+m \!\!\!\! \mod 2)} w^k \\
z^{-(j+m \!\!\!\! \mod 2)} w^k & -(1 - \varepsilon^2)^{1/2} e^{- 2
\pi i k m \alpha}
\end{pmatrix} \right\|.
\end{align*}
The $w$-dependence of the matrix in the last expression is
analytic and hence the $\log$ of its norm is subharmonic.
Therefore,
\begin{align*}
\int_\Omega \log \| A^z_n(\theta,j) \| \, d\mu(\theta,j) & = \frac12
\int_\T \log \| A^z_n(\theta,0) \| \, d\theta + \frac12
\int_\T \log \| A^z_n(\theta,1) \| \, d\theta \\
& \ge n \log \frac{(1 -
\varepsilon^2)^{\frac{1}{2}}}{\varepsilon}.
\end{align*}
Since
$$
\gamma(z) = \lim_{n \to \infty} \frac{1}{n} \int_\Omega \log \|
A^z_n(\theta,j)\| \, d\mu(\theta,j),
$$
the result follows.
\end{proof}

In Theorem~\ref{t.1} we considered functions given by simple
exponentials. Since we obtained explicit terms which bound the
Lyapunov exponents uniformly from below, it is possible to add
small perturbations to the function and retain uniform positivity
of the Lyapunov exponents. For example, given an integer $k \ge 1$
and $\lambda, a_{-k} , \ldots, a_{k-1} \in \C$, we set
\begin{equation}\label{fpertdef}
f_\lambda(\theta,j) = \begin{cases} (1 - \varepsilon^2)^{1/2}
\left( e^{2 \pi i k \theta} + \lambda \sum_{l = - k}^{k-1} a_l
e^{2 \pi i l \theta} \right) & j = 0, \\
(1 - \varepsilon^2)^{1/2} \left( e^{- 2 \pi i k \theta} + \lambda
\sum_{l = - k}^{k-1} a_l e^{- 2 \pi i l \theta} \right) & j = 1.
\end{cases}
\end{equation}
Since we need $f_\lambda$ to take values in $\D$, we have to
impose an upper bound on the admissible values of $\lambda$.
Clearly, once $\varepsilon \in (0,1)$, $k \ge 1$ and the numbers
$a_l \in \C$ are chosen, there is $\lambda_0 > 0$ such that for
$\lambda$ with modulus bounded by $\lambda_0$, the range of
$f_\lambda$ is contained in $\D$.

\begin{theorem}\label{t.2}
Let $(\Omega,\mu,T)$ be as above. For every $\varepsilon \in
(0,\frac{1}{\sqrt{2}})$, $k \in \Z_+$, and $\{ a_l \}_{l =
-k}^{k-1} \subset \C$, there is $\lambda_1 > 0$ such that for
every $\lambda$ with $|\lambda| < \lambda_1$, there is $\gamma_-
> 0$ for which the Lyapunov exponent $\gamma(\cdot)$ associated with $f_\lambda$
given by \eqref{fdef} satisfies
$$
\inf_{z \in \partial \D} \gamma(z) \ge \gamma_- .
$$
\end{theorem}

\begin{proof}
The smallness condition $|\lambda| < \lambda_1$ needs to address
two issues. First, the range of the function $f_\lambda$ must be
contained in $\D$, so we need $\lambda_1 \le \lambda_0$. Second,
the explicit strictly positive uniform lower bound obtained in the
proof of Theorem~\ref{t.1} for the case $\lambda = 0$ changes
continuously once the perturbation is turned on. Thus, it remains
strictly positive for $|\lambda|$ small enough. Notice that the
degree requirements for the perturbation in \eqref{fpertdef} are
such that the subharmonicity argument from the proof of
Theorem~\ref{t.1} goes through without any changes.
\end{proof}

\section{Discussion}

In the previous section we proved a uniform lower bound for the
Lyapunov exponents associated with strongly coupled almost
periodic sequences of Verblunsky coefficients. A few remarks are
in order.

The Verblunsky coefficients take values in the open unit disk and
the unit circle has to be regarded as the analogue of infinity in
the Schr\"odinger case. Thus, just as the coupling constant is
sent to infinity in the application of Herman's argument in the
Schr\"odinger case, the coupling constant is sent to one in our
study. Notice that we need a rather uniform convergence to the
unit circle, whereas one may have zeros in the Schr\"odinger case.
In particular, while Herman's argument applies to all non-constant
trigonometric polynomials in the Schr\"odinger case, we only treat
small perturbations of simple exponentials.

Another limitation of our proof is that it requires the
consideration of the product $\T \times \Z_2$. It would be nicer
to have genuine quasi-periodic examples, that is, generated by
minimal translations on a finite-dimensional torus. Our attempts
to produce such examples have run into trouble with analyticity
issues. It would be of interest to produce quasi-periodic examples
or to demonstrate why Herman's argument cannot work for any of
them.

As explained by Simon in \cite[Theorem~12.6.1]{S2}, as soon as one
knows that $\gamma(z)$ is positive for (Lebesgue almost) every $z
\in \D$, one can immediately deduce that for $\mu$-almost all
elements of $\Omega$, Lebesgue almost all Aleksandrov measures
associated with the sequence of Verblunsky coefficients in
question are pure point. This is applicable to our examples for
$\varepsilon \in (0,\frac{1}{\sqrt{2}})$ and $|\lambda|$ small
enough.


\begin{thebibliography}{10}

\bibitem{B} J.\ Bourgain, \textit{Green's Function Estimates for Lattice
Schr\"odinger Operators and Applications}, Annals of Mathematics
Studies, 158, Princeton University Press, Princeton (2005)

\bibitem{CFKS} H.\ Cycon, R.\ Froese, W.\ Kirsch, B.\ Simon, \textit{Schr\"odinger Operators
with Application to Quantum Mechanics and Global Geometry}, Texts
and Monographs in Physics, Springer-Verlag, Berlin (1987)

\bibitem{GHMT} F.\ Gesztesy, H.\ Holden, J.\ Michor, G.\ Teschl, Local conservation laws and the
Hamiltonian formalism for the Ablowitz-Ladik hierarchy, to appear
in \textit{Stud.\ Appl.\ Math.}

\bibitem{H} M.\ Herman, Une m\'{e}thode pour minorer les exposants de Lyapunov
et quelques exemples montrant le caract\`{e}re local d'un
th\'{e}or\`{e}me d'Arnold et de Moser sur le tore de dimension
$2$, \textit{Comment.\ Math.\ Helv} {\bf 58} (1983), 453--502

\bibitem{S} B.\ Simon, \textit{Orthogonal Polynomials on the Unit Circle. Part~1. Classical
Theory}, Colloquium Publications, 54, American Mathematical
Society, Providence (2005)

\bibitem{S2} B.\ Simon, \textit{Orthogonal Polynomials on the Unit Circle. Part~2. Spectral Theory},
Colloquium Publications, 54, American Mathematical Society,
Providence (2005)

\end{thebibliography}
\end{document}